\documentclass[11pt, a4paper]{article}
\usepackage{amsfonts}
\usepackage{amsbsy}
\usepackage{amssymb}
\usepackage{amsmath}
\usepackage{amsmath,amscd}
\usepackage{amsthm}
\usepackage{fancyhdr}

\theoremstyle{plain}
\newtheorem{theorem}{Theorem}[section]

\newtheorem{corollary}{Corollary}[section]
\newtheorem{proposition}{Proposition}[section]
\theoremstyle{definition}
\newtheorem{definition}{Definition}[section]
\newtheorem{remark}{Remark}[section]
\newtheorem{example}{Example}[section]

\begin{document}

\title{A note on the basic Lichnerowicz cohomology of
transversally locally conformally K\"{a}hlerian foliations}
\author{Cristian Ida}
\date{}
\maketitle
\begin{abstract}
In this paper we generalize the basic Lichnerowicz cohomology on transversally locally conformally K\"{a}hlerian foliations and we study its relation with basic
Bott-Chern cohomology and $0$--th basic Dolbeault cohomology with values in the
associated foliated weight bundle.
\end{abstract}

\medskip 
\begin{flushleft}
\strut \textbf{2000 Mathematics Subject Classification:} 53C12, 53C55, 57R30, 32C35, 55N30.  

\textbf{Key Words:} transversally locally conformally K\"{a}hlerian foliation, Lichnerowicz
cohomology, Bott-Chern cohomology. 
\end{flushleft}

\section{Introduction and preliminaries}
\setcounter{equation}{0}
\subsection{Introduction}
A locally conformally K\"{a}hler (LCK) manifold $M$ is a complex manifold whose universal cover $\widetilde{M}$ has a K\"{a}hler metric $g$ such that $\pi_1(M)$  acts on $(\widetilde{M},g)$ holomorphically and conformally. The fundamental properties of LCK manifolds were studied by Vaisman, Kashiwada, Dragomir,  Ornea and  Verbitsky.

The Lichnerowicz cohomology, also known in literature as Morse-Novikov cohomology, is a cohomology defined for a smooth manifold $M$ and a closed $1$--form $\theta$. It is defined by twisting the usual differential of the de Rham complex $\Omega^{\bullet}(M)$ of $M$; namely, the Lichnerowicz cohomology is the cohomology of a complex $(\Omega^{\bullet}(M),d_{\theta})$, where $d_{\theta}$ is defined by $d_{\theta}\varphi=d\varphi-\theta\wedge\varphi$. This cohomology was originally defined by Lichnerowicz and Novikov in the context of Poisson geometry and Hamiltonian mechanics, respectively.
Lichnerowicz cohomology is naturally defined for a LCK manifold with its canonical closed $1$--form called the Lee form, \cite{O-V, Va5}. Using the complex structures, variants of Lichnerowicz cohomology are defined for the Dolbeault cohomology and the Bott-Chern cohomology. These can be considered as fundamental invariants of LCK manifolds as established mainly by  Vaisman, Ornea and  Verbitsky.

In a paper by Barletta and Dragomir \cite{B-D} is introduced a new class of foliations called transversally locally conformally K\"{a}hler foliations (transversally LCK foliations), which is a foliated version of LCK manifolds roughly in the following sense: This class of foliations has a LCK structure on the direction transverse to the leaves. For instance, the simple foliation defined by a $C^{\infty}$ submersion $f:\mathcal{M}\rightarrow M$ of $\mathcal{M}$ onto a LCK manifold $M$ is transversally LCK. The case where the dimension of the leaves is zero corresponds to the original LCK manifolds.

The aim of this note is to extend some theories related to Lichnerowicz cohomology and its variants to basic forms on transversally LCK foliations. In this sense, in the preliminary subsection following \cite{EK1, EK-G}, we make a short review on the de Rham and Dolbeault theory for basic forms on transversally (holomorphic) foliations. The second section is dedicated to study of the basic Lichnerowicz cohomology of transversally LCK foliations. The relation of this cohomology with basic Bott-Chern cohomology and $0$--th
basic Dolbeault cohomology with values in the foliated weight bundle of a transversally LCK foliation is also studied, obtaining a basic version of some known results in the case of LCK manifolds due to  Ornea and  Verbitski \cite{O-V}.

\subsection{Preliminaries}
Let us consider $\mathcal{M}$ an $(n + m)$--dimensional manifold which will be assumed to be connected and orientable. Differential forms (and in particular functions) will take their values in the field of complex numbers $\mathbb{C}$. If $\varphi$ is a form, then $\overline{\varphi}$ denote its complex conjugate and we say that $\varphi$ is \textit{real} if $\varphi=\overline{\varphi}$.
\begin{definition}
A codimension $n$ foliation $\mathcal{F}$ on $\mathcal{M}$ is defined by a foliated cocycle $\{U_i,\varphi_i,f_{i,j}\}$ such that:
\begin{enumerate}
\item[(i)] $\{U_i\}$, $i\in I$ is an open covering of $\mathcal{M}$;
\item[(ii)] For every $i\in I$, $\varphi_i:U_i\rightarrow M$ are submersions, where $M$ is an $n$--dimensional
manifold;
\item[(iii)] The maps $f_{i,j}:\varphi_i(U_i\cap U_j)\rightarrow \varphi_j(U_i\cap U_j)$ satisfy
\begin{equation}
\varphi_j=f_{i,j}\circ\varphi_i
\label{I1}
\end{equation}
for every $(i,j)\in I\times I$ such that $U_i\cap U_j\neq\emptyset$.
\end{enumerate}
\end{definition}
Every fibre of $\varphi_i$ is called a \textit{plaque} of the foliation. Condition ({\ref{I1}}) says that, on the intersection $U_i\cap U_j$ the plaques defined respectively by $\varphi_i$ and $\varphi_j$ coincides. The manifold $\mathcal{M}$ is decomposed into a family of disjoint immersed connected submanifolds of dimension $m$; each of these submanifolds is called a \textit{leaf} of $\mathcal{F}$.

We say that $\mathcal{F}$ is \textit{transversally orientable} if on $M$ can be given an orientation which is preserved by all $f_{i,j}$. By $T\mathcal{F}$ we denote the tangent bundle to $\mathcal{F}$ and $\Gamma(\mathcal{F})$ is the space of its global sections i.e. vector fields tangent to $\mathcal{F}$. We say that a differential form $\varphi$ is \textit{basic} if it satisfies $i_{X}\varphi=\mathcal{L}_{X}\varphi=0$ for every $X\in\Gamma(\mathcal{F})$, where $i_X$ and $\mathcal{L}_X$ denote the interior product and Lie derivative with respect to $X$, respectively.  A \textit{basic function} is a function constant on the leaves; such functions form an algebra denoted by $\mathcal{F}_{b}(\mathcal{M})$. The quotient $Q\mathcal{F}=T\mathcal{M}/T\mathcal{F}$ is the normal bundle of $\mathcal{F}$. A vector field $Y\in\mathcal{X}(\mathcal{M})$ is said to be \textit{foliated} if, for every $X\in\Gamma(\mathcal{F})$ we have $[X,Y]\in\Gamma(\mathcal{F})$; $\mathcal{X}(\mathcal{M},\mathcal{F})$ denotes the algebra of foliated vector fields on $\mathcal{M}$. The quotient $\mathcal{X}(\mathcal{M}/\mathcal{F})=\mathcal{X}(\mathcal{M},\mathcal{F})/\Gamma(\mathcal{F})$ is called the algebra of \textit{basic vector fields} on $\mathcal{M}$.

In  this paper a system of local coordinates adapted to the foliation $\mathcal{F}$ means coordinates $(z^1,\ldots,z^n,y^1,\ldots,y^m)$ on an open subset $U$ on which the foliation is defined by the equations $dz^i=0,\,i=1,\ldots,n$. If $\mathcal{F}$ is transversally holomorphic (see
Defnition 1.2.2 below) $z^1,\ldots,z^n$ will be complex coordinates.

We recall that a transverse structure to $\mathcal{F}$ is a geometric structure on $M$ invariant by all the local diffeomorphisms $f_{i,j}$. Such a transverse structure can be considered as a geometric structure on the leaf space $\mathcal{M}/\mathcal{F}$ (which is not a manifold in general).
\begin{definition}
\begin{enumerate}
\item[1.2.1.] If $M$ is a Riemannian manifold and all the $f_{i,j}$ are isometries then $\mathcal{F}$ is said
to be Riemannian. This means that the normal bundle $Q\mathcal{F}$ is equipped with a Riemannian metric which is "invariant along the leaves".
\item[1.2.2.] If $M$ is a complex manifold and all the $f_{i,j}$ are biholomorphic maps then we say
that $\mathcal{F}$ is transversally holomorphic. In that case, any transversal to $\mathcal{F}$ inherits a
complex structure.
\item[1.2.3.] If $M$ is a Hermitian manifold and all the $f_{i,j}$ preserve the Hermitian structure
then we say that $\mathcal{F}$ is Hermitian. (The $f_{i,j}$ are in particular biholomorphic maps
and isometries.) The normal bundle $Q\mathcal{F}$ is equipped with
a Hermitian metric "invariant along the leaves".
\item[1.2.4.] If $M$ is a K\"{a}hlerian manifold and all the $f_{i,j}$ preserve the K\"{a}hler structure we say that $\mathcal{F}$ is transversally K\"{a}hlerian. In particular such a foliation is Hermitian. This
is equivalent to the existence of a Hermitian metric $g$ on the normal bundle $Q\mathcal{F}$
which can be writen in a transverse local system of coordinates $(z^1,\ldots,z^n)$ in the
form $g=g_{j\overline{k}}(z,\overline{z})dz^j\otimes d\overline{z}^k$ such that its skew-symmetric part $\omega=\frac{i}{2}g_{j\overline{k}}(z,\overline{z})dz^j\wedge d\overline{z}^k$ is closed ($\omega$ is a basic $2$--form called the basic K\"{a}hler form of $\mathcal{F}$).
\end{enumerate}
\end{definition}
Throughout this paper we consider $\mathcal{F}$ to be transversally holomorphic with $2n$ codimension. Let $\Omega^r(\mathcal{M}/\mathcal{F})$ be the space of all basic forms of degree $r$. It is easy to see that the exterior derivative of a basic form is also a basic form. Indeed, if $\varphi\in\Omega^r(\mathcal{M}/\mathcal{F})$ then $i_{X}\varphi=\mathcal{L}_{X}\varphi=0$ for any $X\in\Gamma(\mathcal{F})$ and, then by Cartan's formulas $\mathcal{L}_X=i_Xd+di_X$ and $d^2=0$ it follows that $i_Xd\varphi=\mathcal{L}_Xd\varphi=0$ for any $X\in\Gamma(\mathcal{F})$. Let us denote by $d_b=d|_{\Omega^{\bullet}(\mathcal{M}/\mathcal{F})}$ the restriction of exterior derivative to basic forms. Then we have $d_b:\Omega^{\bullet}(\mathcal{M}/\mathcal{F})\longrightarrow\Omega^{\bullet+1}(\mathcal{M}/\mathcal{F})$ and the differential complex 
\begin{equation}
0\longrightarrow\Omega^0(\mathcal{M}/\mathcal{F})\stackrel{d_b}{\longrightarrow}\Omega^1(\mathcal{M}/\mathcal{F})\stackrel{d_b}{\longrightarrow}\ldots\stackrel{d_b}{\longrightarrow}\Omega^{2n}(\mathcal{M}/\mathcal{F})\longrightarrow0
\label{I2}
\end{equation}
which is called the \textit{basic de Rham complex} of $\mathcal{F}$; its cohomology is the basic de Rham cohomology $H^{\bullet}(\mathcal{M}/\mathcal{F})$. Now, we consider $Q_{\mathbb{C}}\mathcal{F}=Q\mathcal{F}\otimes_{\mathbb{R}}\mathbb{C}$ the complexified normal bundle of $\mathcal{F}$. Let $J$ be the automorphism of $Q_{\mathbb{C}}\mathcal{F}$ associated to the complex structure; $J$ satisfies $J^2=-{\rm Id}$ and then has two eigenvalues $i$ and $-i$ with associated eigensubbundles respectively denoted by $Q^{1,0}\mathcal{F}$ and $Q^{0,1}\mathcal{F}=\overline{Q^{1,0}\mathcal{F}}$. We have a splitting $Q_{\mathbb{C}}\mathcal{F}=Q^{1,0}\mathcal{F}\oplus Q^{0,1}\mathcal{F}$ which gives rise to decomposition
\begin{displaymath}
\bigwedge^r(Q_{\mathbb{C}}^*\mathcal{F})=\bigoplus_{p+q=r}\bigwedge^{p,q}(Q_{\mathbb{C}}^*\mathcal{F}),
\end{displaymath}
where $\bigwedge^{p,q}(Q_{\mathbb{C}}^*\mathcal{F})=\bigwedge^p(Q^{1,0*}\mathcal{F})\otimes\bigwedge^q(Q^{0,1*}\mathcal{F})$. Basic sections of $\bigwedge^{p,q}(Q_{\mathbb{C}}^*\mathcal{F})$ are called \textit{basic forms of type $(p,q)$} on $(\mathcal{M},\mathcal{F})$. They form a vector space denoted by $\Omega^{p,q}(\mathcal{M}/\mathcal{F})$. We have
\begin{equation}
\Omega^{r}(\mathcal{M}/\mathcal{F})=\bigoplus_{p+q=r}\Omega^{p,q}(\mathcal{M}/\mathcal{F}).
\label{I3}
\end{equation}
As in the classical case of a complex manifold, see \cite{M-K}, the basic exterior derivative is decomposed into two operators
\begin{displaymath}
\partial_b:\Omega^{p,q}(\mathcal{M}/\mathcal{F})\rightarrow\Omega^{p+1,q}(\mathcal{M}/\mathcal{F})\,,\,\overline{\partial}_b:\Omega^{p,q}(\mathcal{M}/\mathcal{F})\rightarrow\Omega^{p,q+1}(\mathcal{M}/\mathcal{F}).
\end{displaymath}
We have $\partial_b^2=\overline{\partial}_b^2=0$ and $\partial_b\overline{\partial}_b+\overline{\partial}_b\partial_b=0$. The differential complex
\begin{equation}
0\longrightarrow\Omega^{p,0}(\mathcal{M}/\mathcal{F})\stackrel{\overline{\partial}_b}{\longrightarrow}\Omega^{p,1}(\mathcal{M}/\mathcal{F})\stackrel{\overline{\partial}_b}{\longrightarrow}\ldots\stackrel{\overline{\partial}_b}{\longrightarrow}\Omega^{p,n}(\mathcal{M}/\mathcal{F})\longrightarrow0
\label{I4}
\end{equation}
is called the \textit{basic Dolbeault complex} of $\mathcal{F}$; its cohomology $H^{p,\bullet}(\mathcal{M}/\mathcal{F})$ is the basic Dolbeault cohomology of foliation $\mathcal{F}$.

\section{Basic Lichnerowicz cohomology of transversally  locally conformally K\"{a}hlerian foliations}
\setcounter{equation}{0}
\subsection{Basic Lichnerowicz cohomology}
Let $(\mathcal{M},\mathcal{F})$ be a transversally holomorphic foliation and $\theta\in\Omega^1(\mathcal{M}/\mathcal{F})$ be a closed basic $1$-form. Denote by $d_{b,\theta}:\Omega^r(\mathcal{M}/\mathcal{F})\rightarrow\Omega^{r+1}(\mathcal{M}/\mathcal{F})$ the map $d_{b,\theta}=d_b-\theta\wedge$.

Since $d_b\theta=0$, we easily obtain that $d_{b,\theta}^2=0$. The differential complex
\begin{equation}
0\longrightarrow\Omega^0(\mathcal{M}/\mathcal{F})\stackrel{d_{b,\theta}}{\longrightarrow}\Omega^1(\mathcal{M}/\mathcal{F})\stackrel{d_{b,\theta}}{\longrightarrow}\ldots\stackrel{d_{b,\theta}}{\longrightarrow}\Omega^{2n}(\mathcal{M}/\mathcal{F})\longrightarrow0
\label{II1}
\end{equation}
is called the \textit{ basic Lichnerowicz complex} of $(\mathcal{M},\mathcal{F})$; its cohomology groups $H^{\bullet}_{\theta}(\mathcal{M}/\mathcal{F})$ are called the \textit{ basic Lichnerowicz cohomology groups} of $(\mathcal{M},\mathcal{F})$.

This is a basic version of the classical Lichnerowicz cohomology,  motivated by Lichnerowicz's
work \cite{L} or Lichnerowicz-Jacobi cohomology on Jacobi and locally conformal
symplectic manifolds, see \cite{Ba, L-L-M}. We also notice that Vaisman in \cite{Va5} studied it under the name of "adapted cohomology" on locally conformal K\"{a}hler (LCK) manifolds. Some notions concerning to a such basic Lichnerowicz cohomology of real foliations may be found in \cite{Ha}.

We notice that, locally, the basic Lichnerowicz complex becames the basic de Rham complex after a change $\varphi\mapsto e^f\varphi$ with $f$ a basic function which satisfies $d_bf=\theta$,
namely $d_{b,\theta}$ is the unique differential in $\Omega^{\bullet}(\mathcal{M}/\mathcal{F})$ which makes the multiplication by the smooth basic function $e^f$ an isomorphism of cochain basic complexes $e^f:(\Omega^{\bullet}(\mathcal{M}/\mathcal{F}),d_{b,\theta})\rightarrow(\Omega^{\bullet}(\mathcal{M}/\mathcal{F}),d_b)$.
\begin{proposition}
The basic Lichnerowicz cohomology depends only on the basic class of $\theta$. In fact, we have the  isomorphism $H^r_{\theta-d_bf}(\mathcal{M}/\mathcal{F})\approx H^r_{\theta}(\mathcal{M}/\mathcal{F})$.
\end{proposition}
\begin{proof}
Since $d_{b,\theta}(e^f\varphi)=e^fd_{b,\theta-d_bf}\varphi$ it results that the map $[\varphi]\mapsto[e^f\varphi]$ is an isomorphism between $H^r_{\theta-d_bf}(\mathcal{M}/\mathcal{F})$ and $H^r_{\theta}(\mathcal{M}/\mathcal{F})$.
\end{proof}
For the basic Lichnerowicz cohomology, similar basic complexes of Dolbeault and Bott-Chern type can be defined. Taking into account the decomposition $\theta=\theta^{1,0}+\theta^{0,1}$, consider the Hodge components of the basic Lichnerowicz differential $d_{b,\theta}=d_b-\theta\wedge$ as
\begin{equation}
d_{b,\theta}=\partial_{b,\theta}+\overline{\partial}_{b,\theta}\,,\,\partial_{b,\theta}=\partial_{b}-\theta^{1,0}\wedge\,,\,\overline{\partial}_{b,\theta}=\overline{\partial}_{b}-\theta^{0,1}\wedge.
\label{II6}
\end{equation}
The diferential complex
\begin{equation}
\ldots\stackrel{\overline{\partial}_{b,\theta}}{\longrightarrow}\Omega^{p,q-1}(\mathcal{M}/\mathcal{F})\stackrel{\overline{\partial}_{b,\theta}}{\longrightarrow}\Omega^{p,q}(\mathcal{M}/\mathcal{F})\stackrel{\overline{\partial}_{b,\theta}}{\longrightarrow}\ldots
\label{II7}
\end{equation}
is called the \textit{ basic Dolbeault-Lichnerowicz complex} of $(\mathcal{M},\mathcal{F})$; its cohomology groups denoted by $H^{p,\bullet}_{\theta}(\mathcal{M}/\mathcal{F})$ are called the \textit{ basic Dolbeault-Lichnerowicz cohomology groups} of $(\mathcal{M},\mathcal{F})$.

The differential complex
\begin{equation}
\label{II8}
\Omega^{p-1,q-1}(\mathcal{M}/\mathcal{F})\stackrel{\partial_{b,\theta}\overline{\partial}_{b,\theta}}{\longrightarrow}\Omega^{p,q}(\mathcal{M}/\mathcal{F})\stackrel{\partial_{b,\theta}\oplus\overline{\partial}_{b,\theta}}{\longrightarrow}\Omega^{p+1,q}(\mathcal{M}/\mathcal{F})\oplus\Omega^{p,q+1}(\mathcal{M}/\mathcal{F})
\end{equation}
is called the \textit{basic Bott-Chern-Lichnerowicz complex} of $(\mathcal{M},\mathcal{F})$ and its cohomology groups 
\begin{displaymath}
H^{\bullet,\bullet}_{BCL}(\mathcal{M}/\mathcal{F})=\frac{{\rm Ker}\{\Omega^{\bullet,\bullet}\stackrel{\partial_{b,\theta}}{\longrightarrow}\Omega^{\bullet+1,\bullet}\}\cap{\rm Ker}\{\Omega^{\bullet,\bullet}\stackrel{\overline{\partial}_{b,\theta}}{\longrightarrow}\Omega^{\bullet,\bullet+1}\}}{{\rm Im} \{\Omega^{\bullet-1,\bullet-1}\stackrel{\partial_{b,\theta}\overline{\partial}_{b,\theta}}{\longrightarrow}\Omega^{\bullet,\bullet}\}}
\end{displaymath}
are called the \textit{ basic Bott-Chern-Lichnerowicz cohomology groups} of $(\mathcal{M},\mathcal{F})$.

In the end of this subsection we apply some considerations from \cite{Va5} for basic forms and we obtain a relation between a twisted basic cohomology associated to $\theta$ and basic real cohomology of $(\mathcal{M},\mathcal{F})$. For every $\theta$ as above, let us consider now the auxiliary basic operator $\widetilde{d}_b=d_b-\frac{r}{2}\theta\wedge$ where $r$ is the degree of the basic form acted on. We notice that $\widetilde{d}_b$ is an antiderivation of basic differential forms and it is easy to see that $\widetilde{d}_b^2=-\frac{1}{2}\theta\wedge d_{b}$. Then $\widetilde{d}_b$ defines a \textit{twisted basic cohomology} of basic differential forms of $(\mathcal{M},\mathcal{F})$, which is given by
\begin{equation}
H^{\bullet}_{\widetilde{d}_b}(\mathcal{M}/\mathcal{F})=\frac{{\rm Ker}\, \widetilde{d}_b}{{\rm Im}\, \widetilde{d}_b\cap {\rm Ker}\,\widetilde{d}_b}
\label{II3}
\end{equation} 
and is isomorphic to the cohomology  of the basic complex  $(\widetilde{\Omega}^{\bullet}(\mathcal{M}/\mathcal{F}), \widetilde{d}_b)$ consisting of the basic differential forms $\varphi\in\Omega^{\bullet}(\mathcal{M}/\mathcal{F})$ satisfying $\widetilde{d}_b^2\varphi=-\theta\wedge d_b\varphi=0$.

The basic complex $\widetilde{\Omega}^{\bullet}(\mathcal{M}/\mathcal{F})$ admits a basic subcomplex $\Omega_{\theta}^{\bullet}(\mathcal{M}/\mathcal{F})$, namely, the ideal generated by $\theta$. On this subcomplex, $\widetilde{d}_b=d_b$, which means that it is a basic subcomplex of the usual basic de Rham complex of $(\mathcal{M},\mathcal{F})$. Hence, one has the homomorphisms
\begin{equation}
a:H^r(\Omega_{\theta}^{\bullet}(\mathcal{M}/\mathcal{F}))\rightarrow H^r_{\widetilde{d}_b}(\mathcal{M}/\mathcal{F})\,,\,b:H^r(\Omega_{\theta}^{\bullet}(\mathcal{M}/\mathcal{F}))\rightarrow H^r(\mathcal{M}/\mathcal{F},\mathbb{R}).
\label{II4}
\end{equation}
Now, we can easily construct a homomorphism
\begin{equation}
c:H^r_{\widetilde{d}_b}(\mathcal{M}/\mathcal{F})\rightarrow H^{r+1}(\mathcal{M}/\mathcal{F},\mathbb{R}).
\label{II5}
\end{equation}
Indeed, if $[\varphi]\in H^r_{\widetilde{d}_b}(\mathcal{M}/\mathcal{F})$, where $\varphi$ is  $\widetilde{d}_b$-closed basic form, then we put $c([\varphi])=[\theta\wedge\varphi]$, and this produces the homomorphism from ({\ref{II5}}). We notice that the existence of $c$ gives some relation between $\widetilde{d}_b$ and the basic real cohomology of $(\mathcal{M},\mathcal{F})$.

\begin{remark}
If we consider the decomposition $\widetilde{d}_b=\widetilde{\partial}_b+\widetilde{\overline{\partial}}_b$ we can construct analogous of homomorphisms $a, b$ and $c$ from \eqref{II4} and \eqref{II5}, respectively, for corresponding basic Dolbeault cohomology.
\end{remark}

\subsection{Basic Lichnerowicz cohomology of transversally LCK foliations}
In this subsection we consider the notion of transversally locally conformally K\"{a}hlerian foliation that is a version of locally conformally K\"{a}hler manifold notion, see \cite{Va3, Va4, Va5}, for transversally K\"{a}hlerian foliations, and we investigate some problems related to basic Lichnerowicz cohomology for such structures. 
\begin{definition}
A locally conformally transversally K\"{a}hlerian foliation, briefly transversally LCK foliation, is a transversally Hermitian foliation $(\mathcal{M},\mathcal{F},g)$ (in the sense of Definition 1.2.3) for which an open covering $\{U_i\}$ of $\mathcal{M}$ exists, and for each $i$ a
basic function $\sigma_i:U_i\rightarrow\mathbb{R}$ such that $\widetilde{g}_i=e^{-\sigma_i}(g|_{U_i})$ is a transverse K\"{a}hler metric on $U_i$ (in the sense of Definition 1.2.4).
\end{definition}
It is easy to see that $\theta|_{U_i}=d_b\sigma_i$ defines a global $d_b$--closed $1$--form, and $(\mathcal{M},\mathcal{F},\omega)$ has the characteristic property \cite{B-D}:
\begin{equation}
d_b\omega=\theta\wedge\omega,
\label{II12}
\end{equation}
where $\omega$ is the basic Hermitian form on $(\mathcal{M},\mathcal{F})$. If we take $U_i=\mathcal{M}$, then $(\mathcal{M},\mathcal{F},\omega)$ is called globally conformally K\"{a}hlerian foliation. The basic form $\theta$ is called the basic Lee form of $(\mathcal{M},\mathcal{F},\omega)$. It is exact iff $(\mathcal{M},\mathcal{F},\omega)$ is globally conformal K\"{a}hlerian foliation.

\begin{example}
A simple foliation defined by a $C^{\infty}$ submersion $f:\mathcal{M}\rightarrow M$  of $\mathcal{M}$ onto a LCK manifold $M$ is transversally LCK foliation. The case where the dimension of the
leaves is zero corresponds to the original LCK manifolds.
\end{example}

\begin{remark}
If $(\mathcal{M},\mathcal{F})$ is a complex analytic foliated manifold, \cite{Va1}, then similarly to Proposition 1.1 from \cite{Va3}, the foliation $\mathcal{F}$ is a transversally LCK foliation if and only if its transverse bundle $Q\mathcal{F}$ has a K\"{a}hler metric which is locally conformally with a
foliated Hermitian metric, or equivalently $(\mathcal{M},\mathcal{F})$ has a Hermitian metric which is locally conformally with a bundle-like metric.
\end{remark}

Now, if $(\mathcal{M},\mathcal{F},\omega)$ is a transversally LCK foliation with basic Lee form $\theta$, then due to ({\ref{II12}}) we have $d_{b,\theta}\omega=0$. Therefore, $\omega$ represents a cohomology class $[\omega]_L$ in the basic Lichnerowicz complex $(\Omega^{\bullet}(\mathcal{M}/\mathcal{F}),d_{b,\theta})$.
\begin{definition}
The basic cohomology class $[\omega]_L\in H^2_\theta(\mathcal{M}/\mathcal{F})$ is called the \textit{basic Lichnerowicz class} of the transversally LCK foliation $(\mathcal{M},\mathcal{F},\omega)$.
\end{definition}
This invariant is a basic version of Morse-Novikov class of LCK manifolds, see \cite{O-V}.

Also, if we consider the decomposition $d_{b,\theta}=\partial_{b,\theta}+\overline{\partial}_{b,\theta}$ we have $\partial_{b,\theta}\omega=\overline{\partial}_{b,\theta}\omega=0$ and so $\omega$ represents a cohomology class $[\omega]_{BCL}$ in the basic Bott-Chern-Lichnerowicz complex of $(\mathcal{M},\mathcal{F},\omega)$.
\begin{definition}
If $(\mathcal{M},\mathcal{F},\omega)$ is a transversally LCK foliation then the cohomology class $[\omega]_{BCL}\in H^{1,1}_{BCL}(\mathcal{M}/\mathcal{F})$ is called the \textit{ basic Bott-Chern-Lichnerowicz class} of $(\mathcal{M},\mathcal{F},\omega)$.
\end{definition}
Thus, for any transversally  LCK  foliation we have three basic cohomological invariants:
\begin{enumerate}
\item[$\bullet$] the basic Lee class $[\theta]\in H^1(\mathcal{M}/\mathcal{F})$;
\item[$\bullet$] the basic Lichnerowicz class $[\omega]_L\in H^2_{\theta}(\mathcal{M}/\mathcal{F})$;
\item[$\bullet$] the basic Bott-Chern-Lichnerowicz class $[\omega]_{BCL}\in H^{1,1}_{BCL}(\mathcal{M}/\mathcal{F})$.
\end{enumerate}
Now, using an argument inspired from \cite{L-L-M1}, we briefly present an another basic cohomology associated to transversally LCK foliations which is connected with the basic Lichnerowicz cohomology of transversally LCK foliations. Let $(\mathcal{M},\mathcal{F},\omega)$ be a transversally LCK foliation with  basic Lee form $\theta$. We consider the basic closed $1$-forms $\theta_0$ and $\theta_1$ defined by
\begin{equation}
\theta_0=m\theta\,\,{\rm and}\,\,\theta_1=(m+1)\theta,\,\,m\in\mathbb{R}.
\label{y1}
\end{equation}
Denote by $H^{\bullet}_{\theta_0}(\mathcal{M}/\mathcal{F})$ and $H^{\bullet}_{\theta_1}(\mathcal{M}/\mathcal{F})$ the basic Lichnerowicz cohomologies of the basic complexes $(\Omega^{\bullet}(\mathcal{M}/\mathcal{F}),d_{\theta_0})$ and $(\Omega^{\bullet}(\mathcal{M}/\mathcal{F}),d_{\theta_1})$, repectively.

Now, let $\widehat{\Omega}^k(\mathcal{M}/\mathcal{F})=\Omega^k(\mathcal{M}/\mathcal{F})\oplus\Omega^{k-1}(\mathcal{M}/\mathcal{F})$ and $\widehat{d}_b:\widehat{\Omega}^k(\mathcal{M}/\mathcal{F})\rightarrow\widehat{\Omega}^{k+1}(\mathcal{M}/\mathcal{F})$ be the basic differential operator defined by
\begin{equation}
\widehat{d}_b(\varphi,\psi)=(d_{b,\theta_1}\varphi-\omega\wedge\psi,-d_{b,\theta_0}\psi).
\label{y2}
\end{equation}
Using ({\ref{II12}}), by direct calculus it follows $\widehat{d}_b^2=0$. Thus, we can consider the basic complex $(\widehat{\Omega}^{\bullet}(\mathcal{M}/\mathcal{F}),\widehat{d}_b)$ and $\widehat{H}^{\bullet}(\mathcal{M}/\mathcal{F})$ the associated basic cohomology. We have the following result which relates $\widehat{H}^{\bullet}(\mathcal{M}/\mathcal{F})$ with basic Lichnerowicz cohomologies $H_{\theta_0}^{\bullet}(\mathcal{M}/\mathcal{F})$ and $H_{\theta_1}^{\bullet}(\mathcal{M}/\mathcal{F})$.
\begin{proposition}
Let $(\mathcal{M},\mathcal{F},\omega)$ be a transversally LCK foliation with basic Lee form $\theta$. Suppose that $i^k:\Omega^k(\mathcal{M}/\mathcal{F})\rightarrow\widehat{\Omega}^k(\mathcal{M}/\mathcal{F})$ and $\pi_2^k:\widehat{\Omega}^k(\mathcal{M}/\mathcal{F})\rightarrow\Omega^{k-1}(\mathcal{M}/\mathcal{F})$ are homomorphisms of $\mathcal{F}_b(\mathcal{M},\mathbb{R})$-modules defined by
\begin{displaymath}
i^k(\varphi)=(\varphi,0) \,\,{\rm and}\,\,\pi_2^k(\varphi,\psi)=\psi,
\end{displaymath}
for $\varphi\in\Omega^k(\mathcal{M}/\mathcal{F})$ and $\psi\in\Omega^{k-1}(\mathcal{M}/\mathcal{F})$. Then:
\begin{enumerate}
\item[i)] The mappings $i^k$ and $\pi_2^k$ induce an exact sequence of basic complexes
\begin{displaymath}
0\longrightarrow(\Omega^{\bullet}(\mathcal{M}/\mathcal{F}),d_{b,\theta_1})\stackrel{i^k}{\longrightarrow}(\widehat{\Omega}^{\bullet}(\mathcal{M}/\mathcal{F}),\widehat{d}_b)\stackrel{\pi_2^k}{\longrightarrow}(\Omega^{\bullet-1}(\mathcal{M}/\mathcal{F}),-d_{b,\theta_0})\longrightarrow0.
\end{displaymath}
\item[ii)] This exact sequence induces a long exact sequence
\begin{displaymath}
\ldots\, H^k_{\theta_1}(\mathcal{M}/\mathcal{F})\stackrel{(i^k)^*}{\longrightarrow}\widehat{H}^k(\mathcal{M}/\mathcal{F})\stackrel{(\pi_2^k)^*}{\longrightarrow}H^{k-1}_{\theta_0}(\mathcal{M}/\mathcal{F})\stackrel{-\delta^{k-1}}{\longrightarrow}H^{k+1}_{\theta_1}(\mathcal{M}/\mathcal{F})\,\ldots,
\end{displaymath}
where the connecting homomorphism $-\delta^{k-1}$ is defined by
\begin{equation}
(-\delta^{k-1})[\varphi]=[\varphi\wedge\omega],\,\,\forall\,[\varphi]\in H^{k-1}_{\theta_0}(\mathcal{M}/\mathcal{F}).
\label{y3}
\end{equation}
\end{enumerate}
\end{proposition}
From the above proposition, we obtain
\begin{corollary}
Let $(\mathcal{M},\mathcal{F},\omega)$ be a transversally LCK foliation with  basic Lee form $\theta$ and such that the basic Lichnerowicz cohomology groups $H^k_{\theta_0}(\mathcal{M}/\mathcal{F})$ and $H^k_{\theta_1}(\mathcal{M}/\mathcal{F})$ have finite dimension, for all $k$. Then, the basic cohomology group $\widehat{H}^k(\mathcal{M}/\mathcal{F})$ has also finite dimension, for all $k$, and
\begin{equation}
\widehat{H}^k(\mathcal{M}/\mathcal{F})\cong\frac{H^k_{\theta_1}(\mathcal{M}/\mathcal{F})}{{\rm Im}\, \delta^{k-2}}\oplus\ker\delta^{k-1},
\label{y4}
\end{equation}
where $\delta^{k}:H^k_{\theta_0}(\mathcal{M}/\mathcal{F})\rightarrow H ^{k+2}_{\theta_1}(\mathcal{M}/\mathcal{F})$ is the homomorphism given by ({\ref{y3}}).
\end{corollary}
\begin{corollary}
Let $(\mathcal{M},\mathcal{F},\omega)$ be a transversally LCK foliation with  basic Lee form $\theta$ such that the dimensions of the basic cohomology groups $H^k_{\theta_0}(\mathcal{M}/\mathcal{F})$ and $H^k_{\theta_1}(\mathcal{M}/\mathcal{F})$ are finite, for all $k$. If the basic Lichnerowicz class $[\omega]_L$ vanish then, for all $k$, we have
\begin{equation}
\widehat{H}^k(\mathcal{M}/\mathcal{F})\cong H^k_{\theta_1}(\mathcal{M}/\mathcal{F})\oplus H^{k-1}_{\theta_0}(\mathcal{M}/\mathcal{F}).
\label{y5}
\end{equation}
\end{corollary}

\subsection{Basic Bott-Chern cohomology of transversally LCK foliations}

In this subsection we present a link between basic Bott-Chern cohomology, basic Lichnerowicz cohomology and $0$-th basic Dolbeault cohomology with values in the associated foliated weight bundle of a transversally LCK foliation $(\mathcal{M},\mathcal{F},\omega)$. The notions are introduced by analogy with the corresponding notions in the case of LCK manifolds, \cite{O-V}.

Let us consider further $(\mathcal{M},\mathcal{F},\omega)$ to be a transversally K\"{a}hlerian foliation. Then the basic K\"{a}hler form $\omega$ determines the basic K\"{a}hler class $[\omega]\in H^{1,1}(\mathcal{M}/\mathcal{F})$, and the difference of basic K\"{a}hler forms which have the same basic K\"{a}hler class is measured by a basic potential $f$:
\begin{displaymath}
\omega_1-\omega=\partial_b\overline{\partial}_bf
\end{displaymath}
see Proposition 3.5.1. from \cite{EK1}. (This also follows from $\partial_b\overline{\partial}_b$-Lemma for transversally
K\"{a}hler foliations, \cite{Noz}). Thus the space of all basic K\"{a}hler structures on a transversally holomorphic foliation $(\mathcal{M},\mathcal{F})$ is locally modeled on $H^{1,1}(\mathcal{M}/\mathcal{F},\mathbb{R})\times(\mathcal{F}_b(\mathcal{M})/{\rm const})$. A similar local description exists for the set of all basic LCK-structures on a given transversally holomorphic foliation, if we fix the basic cohomology class $[\theta]$ of a basic Lee form. Using the basic Bott-Chern-Lichnerowicz class $[\omega]_{BCL}\in H^{1,1}_{BCL}(\mathcal{M}/\mathcal{F})$ of a basic LCK form $\omega$, similarly to \cite{O-V}, we can obtain that the difference of two basic LCK forms in the same basic Bott-Chern-Lichnerowicz class is expressed by a basic potential, just like in transversally K\"{a}hler case, and the set of all basic  LCK structures on a given transversally holomorphic foliation $(\mathcal{M},\mathcal{F})$ is locally parametrized by
\begin{equation}
H^{1,1}_{BCL}(\mathcal{M}/\mathcal{F})\times(\mathcal{F}_b(\mathcal{M})/{\rm Ker}\,d_{b,\theta}d^c_{b,\theta}),
\label{II13}
\end{equation}
where $d_{b,\theta}d^c_{b,\theta}=-2\sqrt{-1}\partial_{b,\theta}\overline{\partial}_{b,\theta}$. 

In order to find a connection between basic Bott-Chern cohomology, basic Lichnerowicz cohomology and $0$--th basic Dolbeault cohomology with values in the associated foliated weight bundle of a transversally LCK foliation $(\mathcal{M},\mathcal{F},\omega)$, we briefly recall some definitions concerning to foliated bundles and basic connections, see \cite{EK-G, Ka-To, Mo}. 

Let $G\hookrightarrow P\rightarrow M$ be a principal bundle with structural group $G\subset {\rm GL}(n,\mathbb{C})$. The group $G$ acts on $P$ on the right and on its Lie algebra $\mathcal{G}$ by the adjoint representation ${\rm Ad}$ i.e., for $g\in G$ and $X\in\mathcal{G}$, ${\rm Ad}_g(X)=gXg^{-1}$. We say that a principal $G$-bundle $P\rightarrow(\mathcal{M},\mathcal{F})$ is a \textit{foliated principal bundle} if it is equipped with a foliation $\mathcal{F}_P$ (\textit{the lifted foliation}) such that the distribution $T\mathcal{F}_P$ is invariant under the right action of $G$, is transversal to the tangent space to the fiber, and projects to $T\mathcal{F}$. A connection $\omega$ on $P$ is called \textit{adapted} to $\mathcal{F}_P$ if the associated horizontal distribution contains $T\mathcal{F}_P$ . An adapted connection $\gamma$ is called a \textit{basic connection} if it is basic as a $\mathcal{G}$-valued form on $(P,\mathcal{F}_P)$.

Let us consider now $E\rightarrow (\mathcal{M},\mathcal{F})$ be a complex vector bundle defined by a cocycle $\{U_i,g_{ij}, G\}$ where $\{U_i\}$ is an open cover of $\mathcal{M}$ and $g_{ij}:U_i\cap U_j\rightarrow G\subset{\rm GL}(n,\mathbb{C})$ are the transition functions. To such a vector bundle we can always associate a principal $G$-bundle $P\rightarrow(\mathcal{M},\mathcal{F})$ whose fibre is the group $G$ and the transition functions
are $g_{ij}$ (viewed as translations on $G$). The complex vector bundle $E\rightarrow (\mathcal{M},\mathcal{F})$ is \textit{foliated} if $E$ is associated to a foliated principal $G$-bundle $P\rightarrow(\mathcal{M},\mathcal{F})$. Let $\Omega^{\bullet}(\mathcal{M},E)$ denote the space of forms on $(\mathcal{M},\mathcal{F})$ with values in $E$. If a connection form $\gamma$ on $P$ is adapted, then
we say that an associated covariant derivative operator  $\nabla$ on $\Omega^{\bullet}(\mathcal{M},E)$ is \textit{adapted} to the foliated bundle. We say that $\nabla$ is a \textit{basic} connection on $E$ if in addition the associated curvature operator $\nabla^2$ satisfies $i_X\nabla^2=0$ for every $X\in\Gamma(\mathcal{F})$.  Note that $\nabla$ is basic if the principal connection $\gamma$ associated to $\nabla$ is basic. Let $\Gamma(E)$ denote the smooth sections of $E$, and let $\nabla$ denote a basic connection on $E$. We say that a section $s:\mathcal{M}\rightarrow E$ is a \textit{basic section} if and only if $\nabla_Xs=0$ for all $X\in\Gamma(\mathcal{F})$. Let  $\nabla_b$ denote the basic connection and $\Gamma_b(E)$ denote the space of basic sections of $E$.

Now, let us consider $E$ to be a foliated complex line bundle over the transversally holomorphic foliation $(\mathcal{M},\mathcal{F})$ with a flat basic connection $\nabla_b$. We denote by $\Omega^{p,q}(\mathcal{M}/\mathcal{F},E)$ the set of all basic $(p,q)$--forms on $(\mathcal{M},\mathcal{F})$ with values in $E$. Consider the basic complex 
\begin{displaymath}
\ldots\longrightarrow\Omega^{p-1,q-1}(\mathcal{M}/\mathcal{F},E)\stackrel{\partial_{b,E}\overline{\partial}_{b,E}}{\longrightarrow}\Omega^{p,q}(\mathcal{M}/\mathcal{F},E)
\end{displaymath}
\begin{equation}
\stackrel{\partial_{b,E}\oplus\overline{\partial}_{b,E}}{\longrightarrow}\Omega^{p+1,q}(\mathcal{M}/\mathcal{F},E)\oplus\Omega^{p,q+1}(\mathcal{M}/\mathcal{F},E)\longrightarrow\ldots,
\label{II15}
\end{equation}
where $\partial_{b,E}$ and $\overline{\partial}_{b,E}$ denote the $(1,0)$ and $(0,1)$-parts of the basic connection operator $\nabla_b:\Omega^{\bullet}(\mathcal{M}/\mathcal{F},E)\rightarrow\Omega^{\bullet+1}(\mathcal{M}/\mathcal{F},E)$. The cohomology of ({\ref{II15}}) denoted by $H^{p,q}_{BC}(\mathcal{M}/\mathcal{F},E)$ is called the \textit{basic Bott-Chern cohomology of $(\mathcal{M},\mathcal{F})$ with values in $E$}. 
\begin{definition}
Let $(\mathcal{M},\mathcal{F},\omega,\theta)$ be a transversally LCK foliation, and $E$ its foliated weight bundle, that is, a trivial complex foliated line bundle with the flat basic connection $d_{b}-\theta$. Consider $\omega$ as a closed basic $(1,1)$-form on $(\mathcal{M},\mathcal{F})$ with values in $E$. Its basic Bott-Chern class $[\omega]_{BC}\in H^{1,1}_{BC}(\mathcal{M}/\mathcal{F},E)$ is called the \textit{ basic Bott-Chern class of the transversally LCK foliation $(\mathcal{M},\mathcal{F},\omega,\theta)$}.
\end{definition}
Now, similarly to \cite{O-V}, we give a characterization of $H^{1,1}_{BC}(\mathcal{M}/\mathcal{F},E)$  in terms of basic Lichnerowicz cohomology of $(\mathcal{M},\mathcal{F},\theta)$ and $0$--th basic Dolbeault cohomology of the foliated weight bundle $E$.

The $0$--th basic Dolbeault cohomology with values in a foliated bundle $E$ can be realized as cohomology of the complex
\begin{equation}
\Gamma_b(E)=\Omega^{0,0}(\mathcal{M}/\mathcal{F}, E)\stackrel{\overline{\partial}_{b,E}}{\longrightarrow}\Omega^{0,1}(\mathcal{M}/\mathcal{F}, E)\stackrel{\overline{\partial}_{b,E}}{\longrightarrow}\Omega^{0,2}(\mathcal{M}/\mathcal{F}, E)\stackrel{\overline{\partial}_{b,E}}{\longrightarrow}\ldots
\label{II16}
\end{equation}
If $E$ is equipped with a flat basic connection, then $\partial_{b,E}:\Omega^{0,1}(\mathcal{M}/\mathcal{F}, E)\rightarrow\Omega^{1,1}(\mathcal{M}/\mathcal{F}, E)$ induces a map
\begin{equation}
H^{0,1}(\mathcal{M}/\mathcal{F},\mathcal{E})\stackrel{\partial^*_{b,E}}{\longrightarrow} H^{1,1}_{BC}(\mathcal{M}/\mathcal{F},E)
\label{II17}
\end{equation}
from the $0$-th basic Dolbeault cohomology with values in the underlying holomorphic foliated bundle (denoted as $\mathcal{E}$) to the basic Bott-Chern cohomology with values in $E$. The basic complex
\begin{equation}
\Gamma_b(E)=\Omega^{0,0}(\mathcal{M}/\mathcal{F}, E)\stackrel{\nabla_b^{1,0}}{\longrightarrow}\Omega^{1,0}(\mathcal{M}/\mathcal{F}, E)\stackrel{\nabla_b^{1,0}}{\longrightarrow}\Omega^{2,0}(\mathcal{M}/\mathcal{F}, E)\stackrel{\nabla_b^{1,0}}{\longrightarrow}\ldots
\label{II18}
\end{equation}
computes the $0$-th basic Dolbeault cohomology with values in holomorphic foliated bundle $\mathcal{E}^{'}$ with a holomorphic structure defined by the complex conjugate of the $\nabla_b^{1,0}$-part of the basic connection. When the foliated bundle $E$ is real, we have $\mathcal{E}\approx\mathcal{E}^{'}$. Then the cohomology of the basic complex ({\ref{II18}}) is naturally identified with $\overline{H^{0,\bullet}(\mathcal{M}/\mathcal{F},\mathcal{E})}$. The map $\overline{\partial}_{b,E}:\Omega^{1,0}(\mathcal{M}/\mathcal{F}, E)\rightarrow\Omega^{1,1}(\mathcal{M}/\mathcal{F}, E)$ defines a homomorphism
\begin{equation}
\overline{H^{0,1}(\mathcal{M}/\mathcal{F},\mathcal{E})}\stackrel{\overline{\partial}^*_{b,E}}{\longrightarrow} H^{1,1}_{BC}(\mathcal{M}/\mathcal{F},E)
\label{II19}
\end{equation}
which is entirely similar to ({\ref{II17}}).

Following step by step the proof of Theorem 4.7. from \cite{O-V}, we obtain an analogous result for basic cohomologies, which allows to compute the basic Bott-Chern cohomology classes in terms of $0$--th basic Dolbeault cohomology and basic Lichnerowicz cohomology.
\begin{theorem}
Let $(\mathcal{M},\mathcal{F})$ be a transversally holomorphic foliation and $E_{\mathbb{R}}$ a trivial real foliated line bundle with flat basic connection $d_b-\theta$, where $\theta$ is a real closed basic $1$-form. Denote by $E$ its complexification, and let $\mathcal{E}$ be the underlying holomorphic bundle. Then there is an exact sequence
\begin{equation}
H^{0,1}(\mathcal{M}/\mathcal{F},\mathcal{E})\oplus\overline{H^{0,1}(\mathcal{M}/\mathcal{F},\mathcal{E})}\stackrel{\partial^*_{b,E}+\overline{\partial}^*_{b,E}}{\longrightarrow}H^{1,1}_{BC}(\mathcal{M}/\mathcal{F},E)\stackrel{\nu}{\longrightarrow}H^2_{\theta}(\mathcal{M}/\mathcal{F}),
\label{II20}
\end{equation}
where $H^2_{\theta}(\mathcal{M}/\mathcal{F})$ is the basic Lichnerowicz cohomology, $\nu$ a tautological map, and the first arrow is obtained as a direct sum of ({\ref{II17}}) and ({\ref{II19}}).
\end{theorem}
From the above theorem, we immediately obtain
\begin{corollary}
Let $(\mathcal{M},\mathcal{F},\omega,\theta)$ be a transversally LCK foliation, $E$ the corresponding flat foliated weight bundle, and $\mathcal{E}$ the underlying holomorphic foliated bundle. Asume that $H^{0,1}(\mathcal{M}/\mathcal{F},\mathcal{E})=0$ and $H^2_{\theta}(\mathcal{M}/\mathcal{F})=0$. Then $H^{1,1}_{BC}(\mathcal{M}/\mathcal{F},E)=0$.
\end{corollary}

\noindent 
Cristian Ida\\
Department of Mathematics and Computer Science,\\
Faculty of Mathematics and Computer Science,\\
Transilvania University of Bra\c{s}ov,\\
Address: Bra\c{s}ov 500091, Str. Iuliu Maniu 50, Rom\^{a}nia.\\
E-mail:\textit{cristian.ida@unitbv.ro}


\begin{thebibliography}{11}

\bibitem{Ba} Banyaga, A., {\em Examples of non $d_{\omega}$-exact locally conformal symplectic forms}. Journal of Geometry \textbf{87}, No 1-2, (2007), 1--13.
\bibitem{B-D} Barletta, E., Dragomir, S., {\em On Transversally Holomorphic Maps of K\"{a}hlerian Foliations}. Acta Aplicandae Mathematicae \textbf{54}, 1998, 121--134.
\bibitem{D-O} Dragomir, S., Ornea, L., {\em Locally conformal K\"{a}hler geometry}. Progress in Math., \textbf{155}, Birh\"{a}user, Boston, Basel, 1998.
\bibitem{EK1} El Kacimi Alaoui, A., {\em Op\'{e}rateurs transversalement elliptiques sur un feuilletage riemannien et applications}. Compositio Math. \textbf{73} (1990), 57--106.
\bibitem{EK-G} El Kacimi Alaoui, A., Gmira, B., {\em Stabilit\'{e} du caract\`{e}re K\"{a}hl\'{e}rien transverse}. Israel J. of Math., \textbf{101} (1997), 323--347.
\bibitem{G-L} Guedira, F., Lichnerowicz, A., {\em Geometrie des algebres de Lie locales de Kirillov}. J.Math. Pures et Appl. \textbf{63} (1984), 407--484.
\bibitem{Ha} Haddou A. Hassan, {\em Foliations and Lichnerowicz Basic Cohomology}. International Math. Forum \textbf{2} (49) (2007), 2437--2446.
\bibitem{Hu} Huybrechts, D., {\em Complex Geometry: An introduction}, Universitext, Springer, 2005.
\bibitem{Ka-To} Kamber, F. W., Tondeur, P., {\em Foliated bundles and characteristic classes}. Lecture Notes in Math. \textbf{493}, Springer-Verlag, Berlin-New York 1975.
\bibitem{Ka} Kashiwada, T., {\em Some properties of locally conformal K\"{a}hler manifolds}. Hokkaido J. Math., \textbf{8} (1979), 191--198.
\bibitem{L-L-M1} de Le\'{o}n, M., L\'{o}pez, B., Marrero, J. C., Padr\'{o}n, E., {\em Lichnerowicz-Jacobi cohomology and homology of Jacobi manifolds : modular class and duality}. available to arXiv:math/9910079v1 [math.DG] 1999.
\bibitem{L-L-M} de Le\'{o}n, M., L\'{o}pez, B., Marrero, J. C., Padr\'{o}n, E., {\em On the computation of the Lichnerowicz-Jacobi cohomology}. J. Geom. Phys. \textbf{44} (2003), 507--522.
\bibitem{L} Lichnerowicz, A., {\em Les vari\'{e}t\'{e}s de Poisson et leurs alg\'{e}bres de Lie associ\'{e}es}. J. Differential Geom. \textbf{12} (2) (1977), 253--300.
\bibitem{Mo} Molino, P, {\em Riemannian foliations}. Progress in Mathematics \textbf{73}, Birkh\"{a}user, Boston 1988.
\bibitem{M-K} Morrow, J., Kodaira, K., {\em Complex Manifolds}. AMS Chelsea Publ., 1971.
\bibitem{Nov} Novikov, S. P., {\em The Hamiltonian formalism and a multivalued analogue of Morse theory}, (Russian). Uspekhi Mat. Nauk \textbf{37} (1982), 3--49.
\bibitem{Noz} Nozawa, H., {\em Deformation of Sasakian metrics}, arXiv:0809.4699v4[DG] 10 Sep 2010.
\bibitem{O-V} Ornea, L., Verbitsky, M., {\em Morse-Novikov cohomology of locally conformally K\"{a}hler manifolds}. J. of Geometry and Physics \textbf{59} (2009), 295--305.
\bibitem{Pa} Pazhitnov, A. V., {\em Exactness of Novikov-type inequalities for the case $\pi_1(M)=\mathbb{Z}^m$ and for Morse forms whose cohomology classes are in general position}. Soviet Math. Dokl. \textbf{39} (3) (1989), 528--532.
\bibitem{Va1} Vaisman, I., {\em From the geometry of hermitian foliate manifolds}. Bull. Math. Soc. Sci. Math. Roumanie, \textbf{17} (1973), 71--100.
\bibitem{Va2} Vaisman, I., {\em Cohomology and diferential forms}. M. Dekker Publ. House, 1973.
\bibitem{Va3} Vaisman, I., {\em Conformal foliations}. Kodai Math. J. \textbf{2} (1979), 26--37.
\bibitem{Va4} Vaisman, I., {\em Locally conformal K\"{a}hler manifolds with parallel Lee form}. Rend. di Mat. Roma, \textbf{12}, (1979), 263--284.
\bibitem{Va5} Vaisman, I., {\em Remarkable operators and commutation formulas on locally conformal K\"{a}hler manifolds}. Compositio Math. \textbf{40} no. 3 (1980), 287--299.
\end{thebibliography}
\end{document}